\documentclass[12pt]{article}

\usepackage{latexsym,amsmath,amsfonts,amsthm}

\usepackage{cite}

\newcommand{\N}{\mathbb{N}}

\newcommand{\R}{\mathbb{R}}
\newcommand{\T}{\mathbb{T}}
\newcommand{\rar}{\mbox{$\rightarrow$}}

\newcommand{\U}{\mathcal{U}}

\newcommand{\Z}{\mathbb{Z}}

\newtheorem{lemma}{Lemma}[section]
\newtheorem{theorem}{Theorem}[section]
\newtheorem{prop}{Proposition}[section]
\newtheorem{cor}{Corollary}[section]

\theoremstyle{remark}

\newtheorem{df}{Definition}[section]
\newtheorem{remark}{Remark}[section]

\begin{document}

\title{Avoidance Control on Time Scales\footnote{To appear in
\emph{J. Optim. Theory Appl.} {\bf 145} (2010), no.~3. In Press.}}

\author{Ewa Paw{\l}uszewicz\thanks{On leave from Bia{\l}ystok
Technical University, Poland. Email: epaw@pb.edu.pl}\\
\texttt{ewa@ua.pt}
\and
Delfim F. M. Torres\thanks{Corresponding author.}\\
\texttt{delfim@ua.pt}}

\date{Department of Mathematics\\
      University of Aveiro\\
      3810-193 Aveiro, Portugal}

\maketitle


\begin{abstract}
We consider dynamic systems on time scales
under the control of two agents. One of the
agents desires to keep the state of the system out of a given set
regardless of the other agent's actions. Leitmann's
avoidance conditions are proved to be valid
for dynamic systems evolving on an arbitrary time scale.
\end{abstract}

\smallskip


\smallskip


\smallskip



\section{Introduction}

The concept of avoidance control was introduced and investigated
by Leitmann and his co-authors in a series of papers
\cite{Leit:72,Leit:81,C:Leit:Sk:2,C:Leit:Sk,Leit:80,LS,LS:83}.
For a recent survey on avoidance control we refer the reader to \cite{Stipanovic}.
The goal of this paper is to initiate
the study of avoidance control on an arbitrary time scale.

Analysis on time scales is nowadays recognized as the right tool
to unify and extend the seemingly disparate fields of discrete-time dynamical
systems and continuous-time dynamical systems \cite{AB,Bh,Bh_Adv}.
In the present paper when the time scale
is fixed to be the real line, then one gets the classical
results of Leitmann and Skowro\'nski \cite{LS};
when the time scale is chosen
to be the integers, one gets analogous results
in the discrete-time case. Moreover,
other models of time can be considered,
and an illustrative example of the results of the paper
is considered for the periodic case.

We consider delta-dynamical systems on time scales under the control
of two agents: the evader and the pursuer.
The evader desires a strategy guaranteing
that no trajectory of the delta-system, emanating
from outside the antitarget, intersects that set,
no matter what strategy the pursuer employs.
Our main objective is to obtain conditions that guarantee
that all the trajectories of a given
linear control system on time scales,
that start outside the prescribed avoidance set,
will never enter the set.

The paper is organized as follows. In Section~\ref{sec:prelim}
we briefly review the necessary calculus on time scales.
In Section~\ref{sec:mf:ts} we prove some results
necessary in order to deal with delta-differential inclusions,
which are considered in Section~\ref{sec:ddi}.
Our main result is proved in Section~\ref{sec:mr}.
The obtained result is then applied
in Section~\ref{sec:lcs:ts} to the case of
linear control systems on time scales.
We finish with Section~\ref{sec:ex}, discussing
a concrete example of application of our results.
Necessary elements of delta-measurability and nonlinear theory
on time scales are presented in Appendix.


\section{Preliminaries}
\label{sec:prelim}

A {\it time scale} $\T$ is an arbitrary nonempty closed subset of
the set $\R$ of real numbers. The standard cases comprise $\T=\R$,
$\T=\Z$, and $\T=h\Z$ for $h>0$. We assume that $\T$ is a topological
space with the topology induced from $\R$.
For $t \in \T$ we define the {\it forward jump operator} $\sigma:\T
\rar \T$ by $\sigma(t):=\inf\{s \in \T:s>t\}$, the {\it backward
jump operator} $\rho:\T \rar \T$ by $\rho(t):=\sup\{s \in \T:s<t\}$,
the {\it graininess function} $\mu:\T \rar [0,\infty)$ by
$\mu(t):=\sigma(t)-t$. Using these operators we can classify
the points of the time scale: if $\sigma(t)>t$, then $t$ is called {\it
right-scattered}; if $\rho(t)<t$, then $t$ is called {\it left-scattered};
if $t<\sup\T$ and $\sigma(t)=t$, then $t$ is called {\it right-dense};
if $t>\inf\T$ and $\rho(t)=t$, then $t$ is {\it left-dense}.
Function $f:\T \rar \R$ is called \emph{rd-continuous} provided it is continuous
at right-dense points in $\T$ and its left-sided limits exist (finite) at left-dense points in $\T$.
We write $f \in C_{\textrm{rd}}$. Function $f:\T \rar \R$ is called
\emph{regulated} provided its right-sided limits exist (finite)
at all right-dense points of $\T$ and its left-sided limits exist (finite) at all left-dense points in $\T$.
Function $f$ is \emph{piecewise rd-continuous}, denoted by $f\in C_{\textrm{prd}}$,
if it is regulated and if it is rd-continuous at all,
except possibility at finitely many, right-dense points $t\in\T$.
Note that composition of a continuous function $g$ with $f\in C_{\textrm{rd}}$ or $f\in C_{\textrm{prd}}$
is respectively rd-continuous or piecewise rd-continuous.
Let $\T^\kappa$ denote the set
\[\T^\kappa:=\left\{ \begin{array}{ccc}
                 \T \setminus (\rho (\sup\T),\sup\T] & {\rm if} &
                  \sup\T<\infty \\
                 \T & {\rm if} & \sup\T=\infty \, .
                \end{array} \right. \]

\begin{df}[\cite{Bh}]
Let $f:\T \rar \R$ and $t \in \T^\kappa$. The {\it delta derivative}
of $f$ at $t$, denoted by $f^\Delta(t)$,
is the real number (provided it exists) with the property
that given any $\varepsilon$ there is a neighborhood
$U=(t-\delta,t+\delta) \cap \T$ (for some $\delta>0$) such that
$|(f(\sigma(t))-f(s))-f^\Delta(t)(\sigma(t)-s)| \leq
\varepsilon|\sigma(t)-s|$ for all $s \in U$. We say that $f$ is {\it
delta differentiable} on $\T^\kappa$ provided $f^\Delta(t)$ exists for all
$t\in \T^\kappa$.
\end{df}

A function $f$ is \emph{rd-continuously delta differentiable} (we
write $f\in C_{\textrm{rd}}^{1}$) if $f^{\Delta}$ exists for all
$t \in \T^{\kappa}$ and $f^{\Delta}\in C_{\textrm{rd}}$.
A continuous function $f$ is \emph{piecewise rd-continuously delta differentiable}
(we write $f\in C_{\textrm{prd}}^{1}$) if $f$ is continuous and $f^{\Delta}$
exists for all, except possibly at finitely many
$t \in \T^{\kappa}$, and $f^{\Delta}\in C_{\textrm{rd}}$.

\begin{remark}
 If $\T=\R$, then for any $t \in \R$ we have $\sigma(t)=t=\rho(t)$
 and the graininess function $\mu(t) \equiv 0$.
 A function $f:\R \rar \R$ is delta differentiable at $t \in \R$ if and only if
 $f^{\Delta}(t)=\lim\limits_{s \rar t}\frac{f(t)-f(s)}{t-s}=f'(t)$, \textrm{i.e.},
 if and only if $f$ is
 differentiable in the ordinary sense at $t$.
 If $\T=\Z$, then for every $t \in \Z$
 we have $\sigma(t)=t+1$, $\rho(t)=t-1$, and
 the graininess function $\mu(t)\equiv 1$.
 A function $f:\Z \rar \R$ is always delta
 differentiable at every $t\in \Z$ with
 $f^{\Delta}(t)=\frac{f(\sigma(t))-f(t)}{\mu(t)}=f(t+1)-f(t)=:\Delta f(t)$.
\end{remark}

\begin{theorem}[Chain Rule \cite{Bh}]
\label{chan_rule}
Let $f:\R\rightarrow \R$ be continuously differentiable and suppose
$g:\T\rightarrow \R$ is delta differentiable.
Then $f\circ g:\T\rightarrow \R$ is delta differentiable and
\[(f\circ g)^{\Delta}(t)=\left\{
\int_0^1f'(g(t)+h\mu(t)g^{\Delta}(t)){\rm d}h \right\}g^{\Delta}(t).\]
\end{theorem}

A continuous function $f:\T \rar \R$ is called \emph{pre-differentiable} with (the region of differentiation) $D$,
provided $D \subset \T^\kappa$, and $\T^\kappa \setminus D$ is countable and contains no right-scattered elements of $\T$.
For any regulated function $f$ there exists a function $F$ that is
pre-differentiable with the region of differentiation $D$ such that $F^{\Delta}(t)=f(t)$ for
all $t \in D$. Such $F$ is called a pre-antiderivative of $f$. Then the \emph{indefinite integral} of $f$ is defined by
$\int f(t)\Delta t:=F(t)+C$, where $C$ is an arbitrary constant. The {\it Cauchy integral} is defined by
\[\int\limits_r^sf(t)\Delta t=F(s)-F(r)\]
for all $r,s \in \T^\kappa$. A function $F:\T \rar \R$ is called an {\it antiderivative} of $f: \T\rar \R$ provided $F^{\Delta}(t)=f(t)$
holds for all $t \in \T^\kappa$. It can be shown that every rd-continuous function has an antiderivative \cite{Bh}.

\begin{remark}
 If $\T=\R$, then $\int\limits_a^b f(\tau) \Delta \tau=\int\limits_a^b f(\tau)d\tau$
 where the integral on the right hand-side is the usual Riemann integral. If $\T=h\Z$, $h>0$,
and $a<b$, then $\int\limits_a^b f(\tau)\Delta\tau=\sum\limits_{t=\frac{a}{h}}^{\frac{b}{h}-1}f(th)h$.
\end{remark}

We say that a function $v:\T\times\R^n \rar \R^n$ is \emph{right-increasing}
(respectively \emph{right-nondecreasing})
at a point $(t_0,x_0)\in\T\setminus\{\max\T\}\times\R^n$ provided:
\begin{enumerate}
\item [(i)] if $t_0$ is right-scattered,
then there is a neighborhood $U_{x_0}$ of point $x_0$ such that
$v(\sigma(t_0),x)>v(t_0,x_0)$
  (respectively $v(\sigma(t_0),x)\geq v(t_0,x_0)$) for any $x \in U_{x_0}$;
\item [(ii)] if $t_0$ is right-dense, then there is a neighborhood $U_{(t_0,x_0)}$
of point $(t_0,x_0)$ such that $v(t,x)>v(t_0,x_0)$ (respectively $v(t,x)\geq v(t_0,x_0)$)
for all $(t,x)\in U_{(t_0,x_0)}$ with $t > t_0$.
\end{enumerate}
Similarly, we say that function $v:\T\times\R^n \rar \R^n$ is \emph{right-decreasing}
(respectively \emph{right-nonincreasing}) if above, in condition (i),
$v(\sigma(t_0),x)<v(t_0,x_0)$ (respectively $v(\sigma(t_0),x)\leq v(t_0,x_0)$),
and in condition (ii) $v(t,x)<v(t_0,x_0)$ (respectively
$v(t,x)\leq v(t_0,x_0)$).

For each function $t \mapsto v(t,x(t)) \in C_{\textrm{rd}}^{1}$, let us define the operator
$\zeta$ by
\begin{equation*}
\zeta(v)(t,x(t)) := v(t+\mu(t),x(t)+\mu(t) x^\Delta(t)) \, ,
\end{equation*}
and operator $\mathcal{D}$ by
\begin{equation}
\label{eq:op:D}
\mathcal{D}(v)(t,x(t)) :=\left\{
             \begin{array}{ll}
               \frac{\zeta(v)(t,x(t))-v(t,x(t))}{\mu(t)} & \hbox{for $\mu(t)\neq 0$;} \\
               \frac{\partial v}{\partial t}(t,x(t))+\frac{\partial v}{\partial x}(t,x(t)) x^\Delta(t) & \hbox{for $\mu(t)=0$.}
             \end{array}
           \right.
\end{equation}
Note that if $\mu(t)\equiv 0$, then
$\mathcal{D}(v)(t,x(t))=\frac{d}{dt}v(t,x(t))$
and $\zeta(v)=v$ for any function $v$.

\begin{prop}
\label{prop.2.5}
Let $t\in\T$.
\begin{enumerate}
\item [(i)] If $\mathcal{D}(v)(t,x(t))\leq 0$, then $v$
is a right-nonincreasing function of $t$.
\item [(ii)] If $\mathcal{D}(v)(t,x(t))\geq 0$,
then $v$ is a right-nondecreasing function of $t$.
\end{enumerate}
\end{prop}
\begin{proof}
We prove $(i)$. Proof of $(ii)$ is similar.
Let us fix $t_0\in\T$. If $t_0$ is right-scattered
and $x\in U_{x_0}$, then $\mathcal{D}(v)(t,x(t))\leq 0$ implies that
\[\frac{v(t_0+\mu(t_0),x_0+\mu(t_0)x^\Delta_{|_{t=t_0}}(t))-v(t,x_0)}{\mu(t_0)}\leq 0.\]
Since $\mu(t_0)$ is positive, then $\sigma(v)(t,x)\leq v(t_0,x_0)$.

Let $t_0$ be a right-dense point and $(t,x)\in U_{(t_0,x_0)}$. Let $\tau=t-t_0$ for every $t\in U_{(t_0,x_0)}$.
Define $x_1(\tau):=t = \tau + t_0$
and $x_2(\tau):=v(t) = v(\tau + t_0)$, and note that the solution $v(t)$
for $t\in [t_0,+\infty)_{\T}$ of the $\Delta$-differential equation
$v^\Delta(t)=v(t,x(t))$, $v(t_0)=x_0$, can be obtain as solution
$x_2(\tau)$, $\tau\geq 0$, of the following system of equations:
\begin{eqnarray*}
   x_1^\Delta(\tau)=1,\;\;\;\;\;\;\;\;\;x_1(0)=t_0\\
   x_2^\Delta(\tau)=v(x_1(\tau),x_2(\tau)),\;\;\;\;\;x_2(0)=x_0.
\end{eqnarray*}
Note that in this case $\mathcal{D}(v)(t,x(t))$ is nothing else
as $\frac{d}{dt}\left(v(x_1(t),x_2(t))\right)$. Since $\mathcal{D}(v)(t,x(t))\leq 0$,
then $\frac{d}{dt}\left(v(x_1(t),x_2(t))\right)\leq 0$
and it follows that function $t \mapsto v(x_1(t),x_2(t))$
is a nonincreasing function of $t$.
\end{proof}


\section{Multifunctions on Time Scales}
\label{sec:mf:ts}

Let $\T$ be an arbitrary time scale. If $t_0, t_1\in\T$ and $t_0 \leq t_1$, then $[t_0,t_1]_{\T}$
denotes the intersection of the real closed interval $[t_0,t_1]$ with $\T$.
Similar notation is used for open, half-open, or infinite intervals.
Additionally, we will assume that interval $[t_0,t_1)_{\T}$ is nonempty.

Function $f$ is said to be \emph{lower rd-semicontinuous at point $t_0\in\T$}
if it is rd-continuous and for any $\varepsilon>0$ there is
a neighborhood $U_{t_0}$ of $t_0$ in which $f(t)>f(t_0)-\varepsilon$.
If $f(t_0)=-\infty$, then we regard $f$ as lower semicontinuous at $t_0$.

\begin{lemma}\label{bouned_semicont}
A finite lower rd-semicontinuous function $f$ defined
on a compact time scale interval  is bounded from below.
\end{lemma}
\begin{proof}
A lower rd-semicontinuous function is a regulated function.
The result follows because every regulated function defined
on a compact interval is bounded \cite{Bh}.
\end{proof}

Let us recall the definition of lexicographical ordering
on $\R^n$. For any two distinct points $x=(x_1,\ldots,x_n)$ and $y=(y_1,\ldots,y_n)$
we write $x\prec y$ if $x_1 < y_1$ or
there exists $i \in \{2,\ldots,n\}$ such that
$x_1 = y_1$, \ldots, $x_{i-1}=y_{i-1}$ and $x_i<y_i$.
Every compact set $K \subseteq \R^n$ has one first point
$\xi=(\xi_1,\ldots,\xi_n)$ with respect to the lexicographical order.

Let $t\in \T$ and let $t\rar F(t)\subset\R^n$ be a multifunction with compact values.
A \emph{lexicographical selection $t\rar\xi(t)\in F(t)$}, where $\xi(t)$
is the first point of the compact set $F(t)$ with respect to the lexicographical order,
can be determined inductively as follows \cite{BPi}.
Let $\{e_1,\ldots,e_n\}$ be the standard basis in $\R^n$. Define
\begin{eqnarray*}
 \begin{array}{c}
  \xi_1=\min_{x\in K}\langle e_1,x \rangle,\;\;\;K_1=\{x\in K,\;\langle e_1,x\rangle=\xi_1\}\, ,\\
  \xi_2=\min_{x\in K_1}\langle e_2,x \rangle,\;\;\;K_2=\{x\in K_1,\;\langle e_2,x\rangle=\xi_2\}\, ,\\
   \vdots\\
  \xi_n=\min_{x\in K_{n-1}}\langle e_n,x \rangle,\;\;\;K_n=\{x\in K_{n-1},\;\langle e_n,x\rangle=\xi_n\} \, .\\
 \end{array}
\end{eqnarray*}
By induction, all sets $K_n\subseteq\cdots\subseteq K_1 \subseteq K$ are compact,
hence the components $\xi_1,\ldots,\xi_n$ are well defined.
The set $K_n$ contains the single point $\xi$, which precedes
all other points of $K$ in the lexicographical order.

\begin{lemma}\label{a.7.2}
Let $\T$ be a time scale and $t\in \T$. Suppose that $t\rar F(t)\subset\R^n$
is a bounded multifunction with closed graph, defined for $t$ in a closed set
$\mathcal{J}_{\T}\subseteq\T$ ($\mathcal{J}_{\T}$ denotes a time scale interval).
Then for each vector $v\in\R^n$ the function
$\varphi(t):=\min_{t\in F(t)}<v,t>$ is lower rd-semicontinuous.
\end{lemma}
\begin{proof}
Proof is similar to the classical one \cite{BPi}. For fixed $\tau \in \mathcal{J}_{\T}$
consider a sequence $\{\tau_k\}$ of points $\tau_k \in \mathcal{J}_{\T}$ with $\tau_k \rar \tau$, such that
\[\lim_{t\rar\tau}\inf_{t\in\mathcal{J}_{\T}}\varphi(t)=\lim_{k\rar\infty}\varphi(\tau_k).\]
Choose $t_k\in F(\tau_k)\subset\R^n$ such that $\varphi(\tau_k)=p\cdot t_k$, where $p$ is
any nonzero real number. Because $F(t)$ is a bounded multifunction with closed graph,
then the obtained sequence of vectors, call them $v$, is uniformly bounded.
So, there exists a convergent subsequence $\{t_{k'}\}$ such that
$t_{k'}\rar\overline{t}$ and, because $F(t)$ has a closed graph, $\overline{t}\in F(t)$.
Hence,
\[\varphi(\tau)=\min_{t\in F(t)}\langle v,t\rangle \leq \langle v,
\overline{t}\rangle=\lim_{k\rar\infty}\varphi(\tau_k)
=\liminf_{t\rar\tau}\mathcal{J}_{\T}.\]
\end{proof}

\begin{theorem}\label{a.7.3}
Let $\T$ be a time scale with $a,b\in\T$, $a<b$. Let $t\in \T$.
Suppose that $t\rar F(t)\subset\R^n$ is a bounded multifunction
with closed graph, defined for $t\in[a,b]_{\T}$.
Let $\xi(t)\in F(t)$ be the lexicographical selection for any $t\in\T$.
Then the map $t \rar \xi(t)$ is $\Delta$-measurable.
\end{theorem}
\begin{proof}
Let $\xi=(\xi_1,\ldots,\xi_n)$. Because $\xi_1=\min_{t\in F(t)}\langle e_1,t \rangle$,
then Lemma~\ref{a.7.2} implies that the first component
$t\rar\xi_1(t)$ of the selection $\xi$ is measurable.
Let us assume that the first $k$ components $\xi_1(\cdot),\ldots,\xi_k(\cdot)$ are all measurable.
Using Proposition~\ref{measure_E}, Proposition~\ref{extension_of_measure},
and Remark~\ref{rem_appendix} in Appendix, we conclude that there exist
countably many disjoint compact sets $\mathcal{J}_{k{_{\T}}}$ such that the
maps $\xi_1,\ldots,\xi_k$ are all rd-continuous when restricted to each $\mathcal{J}_{k{_{\T}}}$.
Let
\[t\rar F_k(t)=\{\tau\in F(t):\;\langle e_l,\tau \rangle=\xi_l(t),\;l=1,\ldots,k\} \, .\]
By construction, sets $F_k(t)$ are all nonempty and compact.
Because multifunction $t\rar F_k(t)$, after restriction to $\mathcal{J}_{k{_{\T}}}$,
has a closed graph, it follows from Lemma~\ref{a.7.2} that function
$\xi_{k+1}(t)=\min_{\tau \in F_k(t)}\langle e_l,\tau \rangle$
is lower rd-semicontinuous. Sets $\mathcal{J}_{k{_{\T}}}$ cover $[a,b]_{\T}$
almost everywhere, so  $\xi_{k+1}(t)$ is $\Delta$-measurable.

By the induction principle with respect to $k$,
all components of $\xi(\cdot)$ are $\Delta$-measurable.
\end{proof}


\section{Delta-Differential Inclusions}
\label{sec:ddi}

Let us consider a control system
\begin{equation}\label{sys1}
 x^\Delta(t)=f(t,x(t),u(t)),\;\;\;u\in\mathcal{U},
\end{equation}
where $\mathcal{U}=\{u(\cdot) \ | \ u(\cdot) \text{ is }\Delta\text{-measurable\;piecewise\;rd-continuous\;with }
u(t)\in U\subset\R^m\;
     \text{for\;all}\;t\in\T\}$
is the set of admissible controls. Suppose that the set
$U\subset\R^m$ of control values
is compact and that function $f:\T\times\Omega \times U \rar \R^n$, $\Omega\subseteq\R^n$,
is rd-continuous with respect to the first variable and continuously differentiable with respect to the second variable $x\in\R^n$.

We say that a rd-continuous function $f:[t_0,t_1]_{\T}\rar\R^n$ is \emph{absolutely rd-continuous}
if for every $\varepsilon>0$ there exists $\delta>0$ such that the following holds
(\textrm{cf.} \cite{CV05}):
if $[s_i,t_i]_{\T}$ is any finite collection of disjoint time scale intervals such that
$\sum_{i=1}^n|t_i-s_i|\leq \delta$, then $\sum_{i=1}^n|f(t_i)-f(s_i)|\leq\varepsilon$.
An absolutely rd-continuous function $x(\cdot)$ defined on some time scale interval
$[t_0,t_1]_{\T}$ is a solution of (\ref{sys1}) if its graph $\{(t,x(t)):t\in[t_0,t_1]_{\T}\}$
is entirely contained in $\Omega$, and if there exists a measurable control $u$,
taking values inside $U$, such that $x^\Delta(t)=f(t,x(t),u(t))$ for almost every $t\in[t_0,t_1]_{\T}$.
Let us recall that by \emph{trajectory} of the system (\ref{sys1})
from $x_0$ corresponding to the control $u\in \mathcal{U}$ we mean the  function
\begin{equation}\label{trajectory}
 x=\psi(t_0,\cdot,x_0,u):[t_0^u,t_1^u]_{\T} \rar \R^n
\end{equation}
such that $x$ is the unique solution of the initial value problem $x^{\Delta}(t)=f_u(t,x(t))$,
$x(t_0)=x_0$, provided it is defined on the interval $[t_0,t_1]_{\T}\subseteq\mathcal{I}_{max}$
($\mathcal{I}_{max}$ described in Theorem \ref{Cauchy}) for all $t \in [t_0,t_1]_{\T}$ and $x(t)\in\Omega$.

In connection with (\ref{sys1}) let us take the multifunction
\begin{equation*}
 F(t,x)=\{f(t,x,\omega):\;\omega\in U\} \, ,
\end{equation*}
where $U\subset\R^m$ is the set of control values,
and consider the $\Delta$-differential inclusion
\begin{equation}\label{inclusion}
x^\Delta(t)\in F(t,x(t)).
\end{equation}
We extend the classical Filippov's theorem \cite{Filippov}
to time scales as follows.

\begin{theorem}\label{Filippov}
An absolutely rd-continuous function $x:[t_0,t_1]_{\T} \rar \R^n$
is a trajectory of (\ref{sys1}) if and only if it satisfies
(\ref{inclusion}) almost everywhere.
\end{theorem}
\begin{proof}
The proof mimics the one given in \cite{BPi} for
the continuous-time case. It is obvious that every solution of (\ref{sys1})
is a solution of (\ref{inclusion}). We prove the reverse implication.
Let $x(\cdot)$ be a solution of (\ref{inclusion}).
For a fixed element $\overline{\omega}\in U$, let
\[W(t)=\left\{
         \begin{array}{ll}
           \{\omega\in U:\;f(t,x(t),\omega)=x^\Delta(t)\} & \hbox{if}\; x^\Delta(t) \in F(t,x(t))\\
           \{\overline{\omega}\} & \hbox{otherwise.}
         \end{array}
       \right.\]
Note that multifunction $W(t)$ is compact. Let us define a control $u$
in such a way that $u(t)$ is the first element of the set $W(t)$ with respect
to the lexicographical order. Note that such control exists since the multifunction
$W(t)$ is compact. Because $W(t)=\overline{\omega}$ holds only on a  set of
$\Delta$-measure zero, \textrm{i.e.}, when $x(\cdot)$ doesn't exist or
$x^\Delta(t)\ne F(t,x(t))$, then $x^\Delta(t)=f(t,x(t),u(t))$
for almost all $t\in[t_0,t_1]_{\T}$. Let us choose a sequence of
disjoint compact intervals $\mathcal{J}_{k{_{\T}}}\subseteq[a,b]_{\T}$,
$k\in\N$, such that $\mu_{\Delta}\left\{[a,b]\setminus\bigcup_{k=1}^{\infty}\mathcal{J}_k\right\}=0$
and the $\Delta$-derivative $x^\Delta$ is well defined and rd-continuous when restricted
to $\mathcal{J}_{k{_{\T}}}$. Hence, the multifunction $W{\mid_{\mathcal{J}_{k{_{\T}}}}}$
has closed graph. Theorem~\ref{a.7.3} implies that for any $t\in\mathcal{J}_{k{_{\T}}}$
the lexicographical selection $t\rar u(t) \in W(t)$ is $\Delta$-measurable.
Since the sets $\mathcal{J}_{k{_{\T}}}$ cover $[a,b]_{\T}$ almost everywhere,
the selection $u:[a,b]_{\T}\rar U$ is $\Delta$-measurable.
\end{proof}


\section{Leitmann's Avoidance Strategies on Time Scales}
\label{sec:mr}

We are now in conditions to extend Leitmann's avoidance strategies
to any time $t$ from an arbitrary time scale $\T$.

Let $p^i :\T\times \R^n \rar U_i$, $i=1,2$, where $U_i$
is a nonempty subset of the space $\R^{d_i}$, be strategies belonging to given
classes of possibility set valued functions $\U_i$
with control values $u^i$ ranging in given sets $U_i$. We allow state- and
time-dependent constrains, \textrm{i.e.}, $U_i=U_i(t,x)$.
Then $u^i\in p^i(t,x) \subseteq U_i \subseteq \R^{d_i}$.

Let us consider a function $f : \T\times \R^n \times \R^{d_1}
\times \R^{d_2} \rar \R^n$ and a set-valued function
\begin{equation*}
 F(t,x):=\{z:z=f(t,x,u^1,u^2)\;\text{where}\;u^i\in p^i(t,x)\} \, .
\end{equation*}

By a \emph{dynamical system} on the time scale $\T$
we will mean the $\Delta$-differential relation
 \begin{equation}\label{dyn_system}
  x^{\Delta}(t) \in F(t,x(t)),\;\;\;x(t_0)=x_0 \, ,
 \end{equation}
$t_0 \in \T$.
From Filippov's theorem on time scales (Theorem~\ref{Filippov})
it follows that, for the given initial condition $x(t_0)=x_0$,
the solution of the $\Delta$-differential inclusion
(\ref{dyn_system}) is an absolute rd-continuous function
\begin{equation}
\label{eq:sol:cs}
 x : [t_0,t_1]_{\T} \rar \R^n,
\end{equation}
satisfying (\ref{dyn_system}) almost everywhere for $t\in[t_0,t_1]_{\T}$.
For $x$ given by \eqref{eq:sol:cs} one has
\begin{equation}
\label{eq:dep:ctr}
\begin{split}
\zeta(v)(t,x(t)) &= v(t+\mu(t),x(t)+\mu(t) x^\Delta(t)) \\
&= v(t+\mu(t),x(t)+\mu(t) f(t,x(t),u^1(t),u^2(t)))
\end{split}
\end{equation}
for some $u^1$ and $u^2$. We will write \eqref{eq:dep:ctr}
as $\zeta(v)(t,x(t),u^1(t),u^2(t))$. Similarly, \eqref{eq:op:D}
will be written as
\begin{multline}
\label{eq:after:FT}
\mathcal{D}(v)(t,x(t),u^1(t),u^2(t)) \\
=\left\{
             \begin{array}{ll}
               \frac{\zeta(v)(t,x(t),u^1(t),u^2(t))-v(t,x(t))}{\mu(t)} & \hbox{for $\mu(t)\neq 0$;} \\
               \frac{\partial v}{\partial t}(t,x(t))+\frac{\partial v}{\partial x}(t,x(t))
               f(t,x(t),u^1(t),u^2(t)) & \hbox{for $\mu(t)=0$.}
             \end{array}
           \right.
\end{multline}

Sometimes it may be convenient to restrict $x$ by $x_{|_\Lambda}$
where $\Lambda$ is an open set or the closure of an open set in $\R^n$.
Then $(t_0,x_0)\in\T\times\Lambda$.

Let $\mathcal{T} \subseteq \Lambda$ denote a set into
which no solution of (\ref{dyn_system}) must enter for some
$p^1(\cdot)\in\U_1$ no matter what $p^2(\cdot)\in\U_2$.
The set $\mathcal{T}$ is called an \emph{antitarget set}. Let us consider two sets:
\begin{enumerate}
  \item [1.] a closed subset $\mathcal{A}\subseteq\Lambda$
  such that $\mathcal{T}\subset\mathcal{A}$;
  \item [2.] the closure $\Lambda_\varepsilon$ of an open subset
  of $\Lambda$ such that $\mathcal{A}\subset\Lambda_\varepsilon$.
\end{enumerate}
The set $\mathcal{A}$ will be called the \emph{avoidance set}
while the set $\Lambda_{\mathcal{A}}:=\Lambda_\varepsilon\setminus\mathcal{A}$
will be called the \emph{safety zone}. Note that the avoidance set can be any set
containing the antitarget set $\mathcal{T}$.

\bigskip

PROBLEM: Determine an \emph{avoidance strategy} $p^1(\cdot)\in\U_1$ such that,
given $(t_0,x_0)\in\T\times\Lambda_{\mathcal{A}}$, no solution of (\ref{dyn_system})
intersects $\mathcal{A}$ no matter what $p^2(\cdot)\in\U_2$.

\bigskip

Denote $\Omega_{\mathcal{A}}=\T^\kappa\times\Lambda_{\mathcal{A}}$.
Then an \emph{attainable set of motions}
$K(t,t_0,x_0)$ from $(t_0,x_0)$ at time $t\in[t_0;\infty)_{\T}$ for the given
$p^1(\cdot)\in\U_1$ is defined in the following way:
\[K(t,t_0,x_0):=\{x(t): \text{given}\; (t_0,x_0)\in\Omega_{\mathcal{A}},\
\text{given}\; p^1(\cdot)\in\U_1\;\text{for\;all}\;p^2(\cdot)\in\U_2\}.\]
The \emph{funnel of motions} $\U_2$ from $(t_0,x_0)$
is defined by
\[K([t_0;\infty)_{\T},t_0,x_0):=\bigcup_{t\in[t_0;\infty)_{\T}}K(t,t_0,x_0),\]
and
\[K(\Omega_{\mathcal{A}},\T):=\bigcup_{(t_0,x_0)\in\Omega_{\mathcal{A}}}K([t_0;\infty)_{\T},t_0,x_0).\]
For the given dynamical system (\ref{dyn_system}) and $\U_i$, $i=1,2$,
a set $\mathcal{A}$ is called \emph{avoidable} if and only if there are
$p^1(\cdot)\in\U_1$ and $\Omega_{\mathcal{A}}\neq\partial\Lambda_\varepsilon
\cap\partial\mathcal{A}\cap{\rm int}\Lambda$ such that
\begin{equation}\label{av}
 K(\Omega_{\mathcal{A}},\T)\cap\mathcal{A}=\partial\Lambda_\varepsilon
 \cap\partial\mathcal{A}\cap{\rm int}\Lambda.
\end{equation}
The condition (\ref{av}) implies that
\[K[\T,\T\times(\Lambda\setminus\mathcal{A})]\cap\mathcal{A}=\partial\Lambda_\varepsilon\cap\partial\mathcal{A}\cap{\rm int}\Lambda\]
and this means global avoidance.

\begin{theorem}\label{th.3.1}
A given set $\mathcal{A}$ is avoidable if there exist
a set $\Omega_{\mathcal{A}}\neq\emptyset$,
a strategy $p^1(\cdot)\in\U_1$, and
a rd-continuous function $V:S\rar\R$, where
$S$ is an open subset of $\bar{\Omega}_{\mathcal{A}}$,
such that for all $(t,x(t))\in\Omega_{\mathcal{A}}$:
 \begin{enumerate}
  \item [(i)] $V(t,x(t))>V(t_1,x^1)$ for all $x^1\in\partial\mathcal{A}$ and $t_1\geq t$ with $t_1,t\in\T^\kappa$;
  \item [(ii)] for all $u^1(t)\in p^1(t,x(t))$, where
  $\mathcal{D}(V)(t,x(t),u^1(t),u^2(t)) \geq 0$ for all $u^2(t)\in U_2$,
 $\tilde{p}^1(\cdot)=p^1(\cdot)_{\mid_{\Omega_{\mathcal{A}}}}$.
\end{enumerate}
\end{theorem}

\begin{proof}
Let $t_0\in\T\setminus\{\max\T\}$. Assume that for some $(t_0,x_0)\in\Omega_{\mathcal{A}}$ there is $t_2\in\T$, $t_2>t_0$,
such that $K(t_0,t_2,x_0)\cap\mathcal{A}\neq \emptyset$. Then, by (i), there is $t_1\in(t_0,t_2]_{\T}$ and
$x^1\in K(t_0,t_2,x_0)\cap\partial\mathcal{A}$ such that $V(t_0,x_0)>V(t_1,x^1)$.
Note that along the trajectories of the system (\ref{dyn_system}) we have
$\mathcal{D}(V)(t,x(t),u(t))$ given in accordance with
\eqref{eq:after:FT}, with $u(t) = (u^1(t),u^2(t))$.
Thus, it is enough to show that for $(t,x)\in\Omega_{\mathcal{A}}$
such that $x\neq x_0$ and $t\neq t_0$, $t\in\T^\kappa$, there is
a $\tau\in\T$ and an admissible control $u$ such that for the trajectory
$\psi(t_0,t,x_0,u)$ of the system $\Lambda$ we have $V(t,\psi(t_0,t,x_0,u))\geq V(t,x)$
for all $t\in(t_0,\tau]_{\T}$ and $V(\tau,\psi(t_0,\tau,x_0,u))> V(\tau,x)$.

Let $(t,x)\in\Omega_{\mathcal{A}}$ be such that $x\neq x_0$ and $t\neq t_0$ with $t\in\T^\kappa$. Let us take two admissible
controls $u^1$ and $u^2$ for the system (\ref{dyn_system}). Then $\psi(t_0,\cdot,x_0,u)$ given by (\ref{trajectory})
is a trajectory corresponding to the control $u=(u^1,u^2)$.
Let us consider the maximal solution of the initial value problem $x^{\Delta}(t)=f_u(t,x(t))$, $x(t_0)=x_0$.
This is defined on some interval $[t_0,\tau)_{\T}$ with $\tau\leq+\infty$. If $\tau=+\infty$, then condition (i)
implies that $V(t,\psi(t_0,t,x_0,u))$ is right-nondecreasing as long as $(t,\psi(t_0,t,x_0,u))$ remains
in the set $\Omega_{\mathcal{A}}$.
For this it is sufficient that it remains
in the set $Z:=\{(t,x):V(t,x)\geq\varepsilon\}$ for some choice of $\varepsilon$.
Assume that there is some $\tilde{t}>0$ in the interval $\mathcal{I}_{\T}$ such that
$V(\tilde{t},\psi(t_0,\tilde{t},x_0,u)) \le \varepsilon$. Because of continuity of $V$,
there is such a first $\tilde{t}$. Thus, we may assume
that $V(t,x(t)) > \varepsilon$ for all $t\in[t_0,\tilde{t})_{\T}$.
So  $t \mapsto V(t,x(t))$ is right-nondecreasing on
$[t_0,\tilde{t}]_{\T}$, which implies that
$V(\tilde{t},\psi(t_0,\tilde{t},x_0,u))\ge V(t,x)>\varepsilon$. We arrived to a contradiction,
and we conclude that $\psi$ remains in the compact set $Z$ for all $t\in\mathcal{I}_{\T}$.
By Proposition~\ref{a.8.11} in Appendix,
$\mathcal{I}_{\T}=[t_0,+\infty)_{\T}$ (as desired). Thus, the trajectory is defined
for all $t\in\T$ and $V$ is right-nondecreasing.
\end{proof}


\section{Linear Control Systems on Time Scales}
\label{sec:lcs:ts}

Let $\T$ be an unbounded time scale.
Let us consider the case of a dynamical system (\ref{dyn_system})
in which its solution $x^\Delta(t)\in f(t,x(t),p^1(t,x(t)),p^2(t,x(t)))$ is linear,
\textrm{i.e.}, described by a linear equation of the form
\begin{equation}
\label{eq:lcs:ts}
 x^{\Delta}(t)=Ax(t)+Bu^1(t)+Cu^2(t)
\end{equation}
where $x\in\Lambda\subseteq\R^n$, $u^i(t)\in U_i\subseteq\R^{d_i}$
for $i=1,2$, and $A,B,C$ are constant matrices. Equation (\ref{eq:lcs:ts})
has a unique forward solution \cite{B:Paw}.

Suppose that matrix $A$ is stable, \textrm{i.e.}, satisfies a Lyapunov equation.
The theory of Lyapunov stability on time scales for linear control systems
is studied in detail in \cite{D2,D1}. Let $Q$ be a symmetric constant matrix
such that $Q\in C_{\textrm{rd}}(\T;\R^{n\times n})$. A quadratic Lyapunov function is given by
\begin{equation}
\label{Lyapunov}
 x^T(t)Qx(t),\;\;\;t\in[t_0;\infty)_{\T},
\end{equation}
with delta derivative
\begin{equation*}
\begin{split}
[x^T(t)Qx(t)]^\Delta &= [Ax(t)]^TQx(t)+x(\sigma(t))QAx(t)\\
&= x^T(t)[A^TQ+QA+\mu(t)A^TQA]x(t).
\end{split}
\end{equation*}
The matrix dynamic equation that is obtained by delta differentiating
(\ref{Lyapunov}) with respect to $t$ is given by
\begin{equation}\label{lyap_eq}
 A^TQ+QA+\mu(t)A^TQA=M,\;\;\;M=M^T.
\end{equation}

\begin{theorem}[\cite{D1}]
If the $n\times n$ matrix $A$ has all eigenvalues
in the corresponding Hilger circle for every
$t\in[t_0,\infty)_{\T}$, then for each $t\in\T$ there exists
some time scale $\mathbb{S}$ such that integration over
$I := [0;\infty)_{\mathbb{S}}$ yields a unique solution
to the Lyapunov matrix equation (\ref{lyap_eq}):
\begin{equation}\label{Q(t)}
 Q(t) =\int_I \mathrm{e}_{A^T}^T(s,0)M \mathrm{e}_A(s,0)\Delta s.
\end{equation}
Moreover, if M is positive definite, then $Q(t)$
is positive definite for all $t\in[t_0,\infty)_{\T}$.
\end{theorem}

Let us choose a matrix $V$ in such a way that $V(t,x)=x^T(t)Qx(t)$.
Then, if the avoidance set $\mathcal{A}=\{x:x^TPx\leq{\rm const}\}$, \textrm{i.e.}, if
it is a ball in $\R^n$, then the condition (i) of
Theorem~\ref{th.3.1} is satisfied. Moreover, let us take a constant matrix $D$ in such a way
that $C=BD$ and define
\begin{equation*}
 U_i:=\{u_i: ||u_i||\leq\alpha_i,\;i=1,2\}
\end{equation*}
with $\alpha_i$ positive constants chosen in such a way that $\alpha_1\geq||D||\alpha_2$,
where $||\cdot||$ denotes the Euclidian norm. Then the condition (ii)
of Theorem~\ref{th.3.1} is met. In this case the avoidance strategy is given by
\begin{equation*}
 \tilde{p}^1(t,x)=\frac{B^TQx}{||B^TQx||}\alpha_1
\end{equation*}
for any $(t,x)$ that does not belong to the set $\T \times \{x:B^TQx=0\}$.
If $(t,x)\in\T \times \{x:B^TQx=0\}$, then $\tilde{u}^1$ may take any
admissible value, \textrm{i.e.}, $\tilde{p}^1(t,x)=U_1$.

\begin{remark}
Let us note that, because of (\ref{Q(t)}), the avoidance strategy depends on the time scale $\T$.
For $\T=\R$ the equation (\ref{lyap_eq}) is nothing else as the classical Lyapunov equation
for the time-invariant system. Similarly for $\T=\Z$.
\end{remark}

If the matrix $-A$ is not stable, but pair $(-A,-B)$ is stabilizable
and graininess function is bounded, then there exist a constant matrix
$K$ such that $-A-KB$ is stable (see \cite{BPW})
and avoidance strategy is of the form
\begin{equation}\label{admissible_strategy}
 \tilde{p}^1(t,x)=Kx+\frac{B^TQx}{||B^TQx||}||D||\alpha_2.
\end{equation}
For the strategy (\ref{admissible_strategy}) to be admissible,
$U_1=U_1(t,x)$ must be such that $\tilde{p}^1(t,x)\subseteq U_1$
for all $(t,x)\in\Omega_{\mathcal{A}}$.


\section{Illustrative Example}
\label{sec:ex}

Let us consider a linear control system
\begin{equation*}
 x^\Delta(t)=\left[
               \begin{array}{cc}
                 0 & 1 \\
                 0 & 0 \\
               \end{array}
             \right]x(t)+\left[
                           \begin{array}{c}
                             0 \\
                             1 \\
                           \end{array}
                         \right]u^1(t)+\left[
                           \begin{array}{c}
                             0 \\
                             1 \\
                           \end{array}
                         \right]u^2(t)
\end{equation*}
with
\[U_2=\{u^2:\;|u^2|\leq 1\}.\]
Note that matrix $-A$ is not stable, but pair $(-A,-B)$ is stabilizable.
Then $K=\left[\begin{array}{cc}
 -1 & 1 \\
\end{array}
\right]$. Let us choose as matrix $M$ the matrix
$\left[\begin{array}{cc}
        1 & 0 \\
        0 & 1 \\
\end{array} \right]$. Then, matrix $Q$,
that solves equation (\ref{lyap_eq}), is of form
\begin{equation}\label{Q}
 Q=\left[
    \begin{array}{cc}
     {\frac{2\,\left(2\,\mu+{\mu}^{2}+3\right)}{6\,\mu+4+3\,{\mu}^{2}+{\mu}^{3}}} & -{\frac {{\mu}^{2}+\mu+2}{6\,\mu+4+3\,{\mu}^{2}+{\mu}^{3}}} \\
     -{\frac {{\mu}^{2}+\mu+2}{6\,\mu+4+3\,{\mu}^{2}+{\mu}^{3}}} & {\frac {{\mu}^{2}+4+\mu}{6\,\mu+4+3\,{\mu}^{2}+{\mu}^{3}}} \\
    \end{array}
  \right]\, .
\end{equation}
The nonlinear part of $\tilde{p}^1(\cdot)$ is
\[{\rm sign}\left(-{\frac {{\mu}^{2}+\mu+2}{6\,\mu+4+3\,{\mu}^{2}+{\mu}^{3}}}x_1+
                   {\frac {{\mu}^{2}+4+\mu}{6\,\mu+4+3\,{\mu}^{2}+{\mu}^{3}}}x_2\right).\]
Provided $\mathcal{A}=\{x:\;x^TQ\leq a, \text{ with } a \text{ any positive real constant}\}$,
any avoidance control is given by
\begin{equation*}
 \tilde{p}^1(t,x)=-x_1+x_2+{\rm sign}\left(-{\frac {{\mu}^{2}+\mu+2}{6\,\mu+4+3\,{\mu}^{2}+{\mu}^{3}}}x_1+
                           {\frac {{\mu}^{2}+4+\mu}{6\,\mu+4+3\,{\mu}^{2}+{\mu}^{3}}}x_2\right)
\end{equation*}
for all
\[(t,x)\notin\left\{(t,x):\;-{\frac {{\mu}^{2}+\mu+2}{6\,\mu+4+3\,{\mu}^{2}+{\mu}^{3}}}x_1+
                        {\frac {{\mu}^{2}+4+\mu}{6\,\mu+4+3\,{\mu}^{2}+{\mu}^{3}}}x_2=0\;\text{and}\;t\in\T\right\}\, ,\]
\textrm{i.e.}, for all
\[(t,x)\notin\{(t,x):\;({\mu}^{2}+\mu+2)x_1=({\mu}^{2}+4+\mu)x_2\;\text{and}\;t\in\T\}.\]
Because the constrain set $U_1$ depends on the choice of the safety zone
$\Delta_{\mathcal{A}}$, one may choose $\Delta_\varepsilon=\{x:\;x^TQx\leq a
+\varepsilon\;\text{for}\;\varepsilon>0\}$, so $U_1$ must be such that
$\tilde{p}^1(t,x)\subseteq U_1$ for all $x\in\Delta_{\mathcal{A}}
=\Delta_\varepsilon\setminus\mathcal{A}$ and $t\in\T$.

Note that if $\T=\R$ then matrix $Q$ given by (\ref{Q}) is nothing else as
\[Q=\left[
    \begin{array}{cc}
     \frac{3}{2} & -\frac{1}{2} \\
     -\frac{1}{2} & 1 \\
    \end{array}
  \right]\]
and avoidance control is given by $\tilde{p}^1(t,x)=-x_1+x_2+{\rm sign}(-\frac{1}{2}x_1+x_2)$
for all $(t,x)\notin\{(t,x):\;x_1=2x_2\;\text{and}\;t\in\R\}$. Hence, for $\T=\R$
we have nothing else as the result of the example given in \cite{LS}.

If $\T=\Z$, then matrix
\[Q=\left[
    \begin{array}{cc}
     \frac{6}{7} & -\frac{2}{7} \\
     -\frac{2}{7} & \frac{3}{7} \\
    \end{array}
  \right].\]
Thus, avoidance control is  $\tilde{p}^1(t,x)=-x_1+x_2+{\rm sign}(-\frac{2}{7}x_1
+\frac{3}{7}x_2)$ for all $(t,x)\notin\{(t,x):\;2x_1=3x_2\;\text{and}\;t\in\Z\}$.

Let us consider now the "impulsive" time scale
$\T=\mathbb{P}_{1,2}:=\bigcup_{k=0}^{\infty}[3k,3k+1]$. The graininess function takes the form
\[\mu(t)=\left\{
           \begin{array}{ll}
             0 & \hbox{if}\; t\in \bigcup_{k=0}^{\infty}[3k,3k+1); \\
             2 & \hbox{if}\;t\in \bigcup_{k=0}^{\infty}\{3k+1\}
           \end{array}
         \right.\]
and
\[Q=\left\{
      \begin{array}{ll}
       \left[\begin{array}{cc}
               \frac{3}{2} & -\frac{1}{2} \\
              -\frac{1}{2} & 1 \\
              \end{array}\right]  & \hbox{if}\; t\in \bigcup_{k=0}^{\infty}[3k,3k+1); \\ \\
       \left[\begin{array}{cc}
               \frac{11}{18} & -\frac{2}{9} \\
              -\frac{2}{9} & \frac{5}{18} \\
             \end{array}\right]  & \hbox{if}\;t\in \bigcup_{k=0}^{\infty}\{3k+1\}.
      \end{array}
    \right.\]
Then, avoidance control is given by
\[\tilde{p}^1(t,x)=\left\{
                     \begin{array}{ll}
                       -x_1+x_2+{\rm sign}(-\frac{1}{2}x_1+x_2) & \hbox{if}\; t\in \bigcup_{k=0}^{\infty}[3k,3k+1); \\
                       -x_1+x_2+{\rm sign}(-\frac{2}{9}x_1+\frac{5}{18}x_2) & \hbox{if}\;t\in \bigcup_{k=0}^{\infty}\{3k+1\}
                     \end{array}
                   \right.\]
for all
\[(t,x)\notin\left\{
               \begin{array}{ll}
                 \{(t,x):\,x_1=2x_2\;\text{and}\;t\in \bigcup_{k=0}^{\infty}[3k,3k+1)\};\\
                 \{(t,x):\,4x_1=5x_2\,\text{and}\;t\in \bigcup_{k=0}^{\infty}\{3k+1\}\}.
               \end{array}
             \right.
\]


\appendix

\section{Appendix}

\subsection*{Elements of $\Delta$-Measures on Time Scales}

The notions of $\Delta$-measurable set and $\Delta$-measurable function
are studied in \cite{CV,dok}.
Let us consider a set $\mathcal{F}=\{[a,b)_{\T}:\;a,b\in\T,\;a\leq b\}$.
The interval $[a,a)_{\T}$ is understood as the empty set.
Let $m_1:\mathcal{F}\rar [0,\infty)$ be a set of functions that assigns
to each interval $[a; b)_{\T}\in \mathcal{F}$ its length: $m_1([a,b)_{\T})=b-a$.

Using the pair $(\mathcal{F};m_1)$, one can generate an outer measure $m_1^*$
on the family of all subsets of $\T$ as follows. Let $E$ be any subset of $\T$.
If there exists at least one finite or countable system of intervals
$V_j\in\mathcal{F}$, $j\in\N$, such that $E\subset\bigcup_j V_j$, then we put
\[m_1^*(E)=\inf\sum_jm_1(V_j)\]
where the infimum is taken over all coverings of $E$ by a finite or countable
system of intervals $V_j \subset \mathcal{F}$. If there is no such covering of $E$,
then we put $m_1^*(E)=\infty$. A set $A\subset\T$ is said
to be \emph{$\Delta$-measurable} if equality
\[m_1^*(E)=m_1^*(E\cap A)+m_1^*(E\cap(\T\setminus A))\]
holds true for any subset $E$ of $\T$. Defining the family
\[\mathcal{M}(m_1^*)=\{A\subset\T:\;A\;\text{is}\;\Delta-\text{measurable}\},\]
the \emph{Lebesgue $\Delta$-measure}, denoted by $\mu_{\Delta}$,
is the restriction of $m_1^*$ to $\mathcal{M}(m_1^*)$.

\begin{prop}[\cite{dok}]
\label{measure_E}
 If set $E$ is Lebesgue measurable, then set $E\cap\T$ is $\Delta$-measurable.
\end{prop}

Let $\tilde{\R}:=[-\infty,+\infty]$. We say that function $f:\T\rar\tilde{\R}$
is $\Delta$-measurable if for every $\alpha\in\R$
the set $f^{-1}([-\infty,\alpha))=\{t\in\T:\;f(t)<\alpha\}$ is $\Delta$-measurable.

\begin{prop}[\cite{dok}]
\label{measure_rd-con}
If f is rd-continuous, then f is $\Delta$-measurable.
\end{prop}

Proposition~\ref{measure_rd-con} and properties of rd-continuous and continuous functions on time scales
implies that if $f$ is a continuous function defined on $\T$, then $f$ is $\Delta$-measurable.
Moreover, if an rd-continuous function $f$ is defined on a $\Delta$-measurable set $E\subseteq\T$,
then $f$ is a $\Delta$-measurable function.

\begin{prop}[\cite{dok}]
Let $f$ be defined on a $\Delta$-measurable subset $E$ of $\T$. Function $f$
is $\Delta$-measurable if the set of all right-dense points
of $E$, where $f$ is discontinuous, is a set of $\Delta$-measure zero.
\end{prop}

Let us define the function
\begin{equation}\label{ext}
 \tilde{f}(t):=\left\{
                  \begin{array}{ll}
                    f(t) & \hbox{if}\;t\in\T \\
                    f(t_i) & \hbox{if}\;t\in(t_i,\sigma(t_i))
                  \end{array}
                \right.
\end{equation}
for some $i \in I \subseteq\N$ and $\{t_i\}_{i\in I}=\{t\in\T:\;t<\sigma(t)\}$.
Then, $f$ is $\Delta$-measurable if and only if $\tilde{f}$ is Lebesgue measurable.

\begin{prop}[\cite{CV}]
\label{extension_of_measure}
Assume that $f:\T\rar\tilde{\R}$ and $\tilde{f}:[a,b]_{\R}\rar\tilde{\R}$
is the extension of $f$ to $[a,b]_{\R}$, defined by (\ref{ext}).
Then $f$ is $\Delta$-measurable if and only if $\tilde{f}$ is Lebesgue measurable.
\end{prop}

Let us recall the classical {\L}uzin's theorem.
\begin{theorem}[\cite{KF}]
\label{Luzin}
Function $f:[a,b]_{\R}\rar\R$ is Lebesgue measurable if and only if,
given $\varepsilon>0$, there is a continuous function $\varphi:[a,b]_{\R}\rar\R$
such that the Lebesgue measure of the set $\{x:f(x)\neq \varphi(x)\}$
is strictly less than $\varepsilon$.
\end{theorem}

As an immediate consequence of Proposition~\ref{extension_of_measure} and {\L}uzin's Theorem~\ref{Luzin},
we obtain an extension of {\L}uzin's theorem to time scales.
\begin{cor}
\label{Luzin_on_TS}
Assume that $f:\T\rar\tilde{\R}$.
Then $f$ is $\Delta$-measurable if and only if, given $\varepsilon>0$,
there is a rd-continuous function $\varphi:[a,b]_{\T}\rar\R$ such that the
$\Delta$-Lebesgue measure of the set $\{x:f(x)\neq \varphi(x)\}$
is strictly less than $\varepsilon$.
\end{cor}

\begin{remark}\label{rem_appendix}
Function $f:[a,b]_{\T}\rar\R^n$
is $\Delta$-measurable if and only if there exists a sequence of disjoint compact subsets
$\mathcal{J}_k\subset[a,b]_{\T}$ with $\mu_{\Delta}\{[a,b]\setminus\bigcup_{k=1}^{\infty}\mathcal{J}_k\}=0$
such that $f_{\mid_{\mathcal{J}_k}}$ is rd-continuous.
\end{remark}


\subsection*{Nonlinear $\Delta$-Differential Equations on Time Scales}

Let us recall \cite{Bh} that function $f:\T \times \R^n \rar \R^n$
is called
\begin{enumerate}
 \item \emph{rd-continuous}, if  $g$ defined by $g(t)=f(t,x(t))$ is
  rd-continuous for any continuous function $x:\T \rar \R^n$;
 \item \emph{bounded} on a set $S \subseteq \T \times \R^n$, if there
  exist constants $m$ and $M$ such that $m \leq f(t,x) \leq M$ for all
  $(t,x) \in S$.
\end{enumerate}

\begin{theorem}[\cite{Bh}]
\label{Cauchy}
 Let $t_0 \in \T$, $x_0 \in \R^n$, $a>0$ with $\inf\T \leq t_0-a$, and $\sup \T \geq t_0+a$.
 Put $\mathcal{I}_a=(t_0-a,t_0+a)_{\T}$ and $\mathcal{V}_b=\{x\in\R^n:|x-x_0|<b\}$.
 Let $f:\T\times\R^n\rar\R^n$ be rd-continuous and regressive. Suppose that for each $(t,x)\in\T\times\R^n$
 there exists a neighborhood $\mathcal{I}_a\times\mathcal{V}_b$ such that $f$ is bounded
 on $\mathcal{I}_a\times\mathcal{V}_b$ and such that the Lipschitz condition
 \[|f(t,x_1)-f(t,x_2)|\leq L(t,x)|x_1-x_2|\,\,\,
 \text{for\,all}\,\,\,(t,x_1),\,(t,x_2)\in\mathcal{I}_a\times\mathcal{V}_b\]
 holds, where $L(t,x)>0$. Then, the initial value problem
 \begin{equation} \label{IVP}
  x^\Delta=f(t,x),\;\;\;x(t_0)=x_0
 \end{equation}
 has exactly one maximal solution $\tilde{\lambda}:\mathcal{I}_{max}\rar\R^n$,
 and the maximal interval of existence $\mathcal{I}_{max}=\mathcal{I}_{max}(t_0,x_0)$ is open.
\end{theorem}

\begin{theorem}[\cite{Bh}]
\label{ext_Couchy}
Suppose that assumptions of Theorem~\ref{Cauchy} are satisfied,
and assume that there exist positive and continuous functions $p$ and $q$ with
\[|f(t,x)|\leq p(t)|x|+q(t)\,\,\,\text{for\,all}\,\,\,(t,x)\in\T\times\R^n.\]
Then each solution of $x^\Delta=f(t,x)$ exists on all of $\T$.
\end{theorem}

\begin{prop}\label{existence_lambda}
 For any compact $\mathcal{K}\subseteq\R^n$ there locally exists a continuous function $\gamma$ such that
 \[|f(t,x)|\leq\lambda(t)\,\,\,\text{for\,all}\,\,\,(t,x)\in\mathcal{I}_a\times\mathcal{K}.\]
\end{prop}
\begin{proof}
 The proof is similar to the one for the real time case (see \cite{S}).
 Let us note that for any $x\in\mathcal{V}_b$ and $t\in\mathcal{I}_a$, $\mathcal{V}_b$ and $\mathcal{I}_a$
 such as in assumptions of Theorem~\ref{Cauchy}, we have
 \[|f(t,x)|\leq|f(t,x_0|+|f(t,x)-f(t,x_0)|\leq p(t)|x|+q(t).\]
 Put $\lambda_{x_0}=p(t)|x|+q(t)$.
 By compactness of the set $\mathcal{K}$ there exist a finite subcover,
 corresponding to sets $\mathcal{V}_{b_1},\ldots,\mathcal{V}_{b_l}$, centered at $x_1,\ldots,x_l$.
Taking $\lambda(t):=\max\{\lambda_{x_1},\ldots,\lambda_{x_l}\}$ we arrive to the intended conclusion.
\end{proof}
Note that function $\lambda$ defined above is integrable on time scales.
\begin{prop}\label{a.8.11}
 Suppose that assumptions of Theorem~\ref{Cauchy} are satisfied, and assume that
 it is known that there is a compact subset $\mathcal{K}\subseteq\R^n$
 such that the maximal solution $\psi$ of (\ref{IVP}) satisfies $\psi(t)\in\mathcal{K}$
 for all $t\in \mathcal{J}_{\T}\subseteq\mathcal{I}_a$ where $\mathcal{J}_{\T}$
 is open relative to $\mathcal{I}_a$. Then,
 \begin{equation*}
  \mathcal{J}_{\T}=[t_0,+\infty)_{\T}\cap\mathcal{I}_a \, .
 \end{equation*}
\end{prop}
\begin{proof}
Proof is the same as the one given for the real time case (see \cite{S}).
\end{proof}


\section*{Acknowledgements}

The authors were supported
by the Centre for Research on Optimization and Control (CEOC)
from the Portuguese Foundation for Science and Technology (FCT),
cofinanced by the European Community fund FEDER/POCI 2010.




\begin{thebibliography}{99}

\bibitem{Leit:72}
R. Aggarwal\ and\ G. Leitmann,
Avoidance control,
Trans. ASME Ser. G. J. Dynamic Systems,
Measurement and Control {\bf 94} (1972), 152--154.

\bibitem{Leit:81}
B. R. Barmish, W. E. Schmitendorf\ and\ G. Leitmann,
A note on avoidance control,
Trans. ASME Ser. G J. Dynamic Systems Measurement Control
{\bf 103} (1981), no.~1, 69--70.

\bibitem{C:Leit:Sk:2}
M. Corless\ and\ G. Leitmann,
Adaptive controllers for avoidance or evasion in an uncertain environment:
some examples, Comput. Math. Appl. {\bf 18} (1989), no.~1-3, 161--170.

\bibitem{C:Leit:Sk}
M. Corless, G. Leitmann\ and\ J. M. Skowronski,
Adaptive control for avoidance or evasion in an uncertain environment,
Comput. Math. Appl. {\bf 13} (1987), no.~1-3, 1--11.

\bibitem{Leit:80}
G. Leitmann, Guaranteed avoidance strategies,
J. Optim. Theory Appl. {\bf 32} (1980), no.~4, 569--576.

\bibitem{LS}
G. Leitmann\ and\ J. Skowro\'nski,
Avoidance control, J. Optimization Theory Appl.
{\bf 23} (1977), no.~4, 581--591.

\bibitem{LS:83}
G. Leitmann\ and\ J. Skowro\'nski,
A note on avoidance control,
Optimal Control Appl. Methods
{\bf 4} (1983), no.~4, 335--342.

\bibitem{Stipanovic}
D. M. Stipanovi\'{c},
A Survey and some new results in avoidance control,
Proceedings of the 15th International Workshop
on Dynamics and Control, IWDC 2009,
J.~Rodellar and E.~Reithmeier (Eds), 166--173.

\bibitem{AB}
R. P. Agarwal\ and\ M. Bohner,
Basic calculus on time scales and some of its applications,
Results Math. {\bf 35} (1999), no.~1-2, 3--22.

\bibitem{Bh}
M. Bohner\ and\ A. Peterson,
{\it Dynamic equations on time scales},
Birkh\"auser Boston, Boston, MA, 2001.

\bibitem{Bh_Adv}
M. Bohner\ and\ A. Peterson,
{\it Advances in dynamic equations on time scales},
Birkh\"auser Boston, Boston, MA, 2003.

\bibitem{BPi}
A. Bressan\ and\ B. Piccoli,
{\it Introduction to the mathematical theory of control},
American Institute of Mathematical Sciences (AIMS), Springfield, MO, 2007.

\bibitem{CV05}
A. Cabada\ and\ D. R. Vivero,
Criterions for absolute continuity on time scales,
J. Difference Equ. Appl. {\bf 11} (2005), no.~11, 1013--1028.

\bibitem{Filippov}
A. F. Filippov, {\it Differential equations with discontinuous
righthand sides}, Translated from the Russian,
Kluwer Acad. Publ., Dordrecht, 1988.

\bibitem{B:Paw}
Z. Bartosiewicz\ and\ E. Paw\l uszewicz,
Realizations of linear control systems on time scales,
Control Cybernet. {\bf 35} (2006), no.~4, 769--786.

\bibitem{D2}
J. J. DaCunha,
Lyapunov stability and Floquet theory for nonautonomous
linear dynamic systems on time scales,
\emph{PhD thesis}, Baylor University, 2004.

\bibitem{D1}
J. J. DaCunha,
Stability for time varying linear dynamic systems on time scales,
J. Comput. Appl. Math. {\bf 176} (2005), no.~2, 381--410.

\bibitem{BPW}
Z. Bartosiewicz, E. Piotrowska\ and\ M. Wyrwas,
Stability, stabilization and observers of linear control systems on time scales,
\emph{Proceedings of the 46th IEEE Conference on Decision and Control}, New Orleans, USA, 2007.

\bibitem{CV}
A. Cabada\ and\ D. R. Vivero, Expression of the Lebesgue $\Delta$-integral
on time scales as a usual Lebesgue integral: application to the calculus
of $\Delta$-antiderivatives, Math. Comput. Modelling {\bf 43} (2006), no.~1-2, 194--207.

\bibitem{dok}
A. Deniz,
Measure theory on time scales,
\emph{MSc thesis}, Graduate School of Engineering
and Sciences of Izmir Institute of Technology, Turkey, 2007.

\bibitem{KF}
A. N. Kolmogorov\ and\ S. V. Fom\=\i n,
{\it Introductory real analysis},
Translated from the second Russian edition
and edited by Richard A. Silverman, Dover, New York, 1975.

\bibitem{S}
E. D. Sontag, {\it Mathematical control theory}, Springer, New York, 1990.

\end{thebibliography}
\end{document}